\documentclass[letterpaper,fontsize=11pt,DIV=10]{scrartcl}

\usepackage[T1]{fontenc}
\usepackage[utf8]{inputenc}

\usepackage{url}
\usepackage[colorlinks,citecolor=blue,urlcolor=blue,linkcolor=blue,linktocpage=true]{hyperref}

\usepackage{amsmath,amsthm,amssymb,bm}
\usepackage{mathtools}
\usepackage{enumitem}
\usepackage{etoolbox}
\usepackage{caption}
\usepackage{microtype}
\usepackage{lmodern}
\usepackage{cleveref}

\numberwithin{equation}{section}
\theoremstyle{plain}
\newtheorem{theorem}{Theorem}[section]
\newtheorem{lemma}[theorem]{Lemma}
\newtheorem{corollary}[theorem]{Corollary}
\newtheorem{proposition}[theorem]{Proposition}
\newtheorem*{corollary*}{Corollary}

\theoremstyle{definition}
\newtheorem{example}[theorem]{Example}
\newtheorem{definition}[theorem]{Definition}
\newtheorem{assumption}[theorem]{Assumption}

\theoremstyle{remark}
\newtheorem{remark}[theorem]{Remark}

\newcommand{\Implies}[2]{$\text{\ref{#1}}\implies\text{\ref{#2}}$}%

\DeclarePairedDelimiter\evaluated{.}{\rvert}
\reDeclarePairedDelimiterInnerWrapper\evaluated{nostarnonscaled}{%
    \mathopen{}#2\mathclose{#3}%
}
\reDeclarePairedDelimiterInnerWrapper\evaluated{star}{%
    \mathopen{}\mathclose\bgroup #1\hskip -\nulldelimiterspace \relax
    #2\aftergroup\egroup #3%
}

\DeclarePairedDelimiterX\event[1]{\{}{\}}{#1}

\DeclareMathOperator{\conv}{conv}

\DeclareMathOperator{\relint}{relint}

\DeclareMathOperator{\effdom}{dom}

\newcommand{\argmin}{\operatornamewithlimits{arg\,min}}

\DeclarePairedDelimiterX\innerp[2]{\langle}{\rangle}{#1,#2}

\providecommand\given{\,|\,}
\newcommand\SetSymbol[1][]{%
  \nonscript\:#1\vert
  \allowbreak
  \nonscript\:
  \mathopen{}}

\DeclarePairedDelimiterX\set[1]\{\}{%
  \renewcommand\given{\SetSymbol[\delimsize]}
  #1 
}

\DeclarePairedDelimiterX\parens[1]\lparen\rparen{%
  \renewcommand\given{\SetSymbol[\delimsize]}
  #1   
}

\DeclarePairedDelimiterX\bracks[1]\lbrack\rbrack{%
  \renewcommand\given{\SetSymbol[\delimsize]}
  #1   
}

\DeclarePairedDelimiterX\braces[1]\lbrace\rbrace{%
  \renewcommand\given{\SetSymbol[\delimsize]}
  #1   
}

\DeclarePairedDelimiterX\angles[1]\langle\rangle{%
  \renewcommand\given{\SetSymbol[\delimsize]}
  #1   
}

\DeclarePairedDelimiter\abs{\lvert}{\rvert}
\DeclarePairedDelimiterX\norm[1]\lVert\rVert{\ifblank{#1}{\:\cdot\:}{#1}}

\makeatletter
\newcommand{\opnorm}{\@ifstar\@opnorms\@opnorm}
\newcommand{\@opnorms}[1]{%
  \left|\mkern-1.5mu\left|\mkern-1.5mu\left|
   #1
  \right|\mkern-1.5mu\right|\mkern-1.5mu\right|
}
\newcommand{\@opnorm}[2][]{%
  \mathopen{#1|\mkern-1.5mu#1|\mkern-1.5mu#1|}
  #2
  \mathclose{#1|\mkern-1.5mu#1|\mkern-1.5mu#1|}
}
\makeatother

\providecommand{\card}[1]{\lvert#1\rvert}

\newcommand{\Integer}{\mathbb{Z}}
\newcommand{\Real}{\mathbb{R}}

\DeclareMathOperator{\vecspan}{span}

\providecommand{\G}{}
\renewcommand{\G}{\mathcal{G}}
\renewcommand{\H}{\mathcal{H}}
\newcommand{\X}{\mathcal{X}}

\newcommand{\PermGroup}{\mathcal{P}}
\newcommand{\OrthoGroup}{\mathcal{O}}

\crefname{assumption}{\textbf{assumption}}{assumptions}
\Crefname{assumption}{Assumption}{Assumptions}
\crefformat{equation}{(#2#1#3)}
\crefrangeformat{equation}{(#3#1#4) to~(#5#2#6)}
\crefformat{equation}{(#2#1#3)}
\crefrangeformat{equation}{(#3#1#4) to~(#5#2#6)}
\crefname{equation}{}{}
\Crefname{equation}{}{}

\crefname{definition}{\textbf{definition}}{definitions}
\Crefname{definition}{Definition}{Definitions}
\crefname{assumption}{\textbf{assumption}}{assumptions}
\Crefname{assumption}{Assumption}{Assumptions}
\crefformat{equation}{(#2#1#3)}
\crefrangeformat{equation}{(#3#1#4) to~(#5#2#6)}

\usepackage[
  backend=biber,
  natbib=true,
  style=authoryear,
  maxbibnames=99,
  giveninits=true,
  uniquename=false
]{biblatex}

\begin{document}

\title{Computational~Sufficiency, Reflection~Groups, and Generalized~Lasso Penalties}
\author{Vincent Q. Vu\\%
Department of Statistics\\%
The Ohio State University}

\maketitle

\begin{abstract}
We study estimators with generalized lasso penalties within the 
computational sufficiency framework introduced by \citet{vu:group}. 
By representing these penalties as support functions of 
zonotopes and more generally Minkowski sums of line segments and rays, 
we show that there is a natural reflection group associated with the 
underlying optimization problem. A consequence of this point of view is 
that for large classes of estimators sharing the same penalty, the 
penalized least squares estimator is computationally minimal sufficient. 
This means that 
all such estimators can be computed by refining the output of any algorithm 
for the least squares case. An interesting technical component is our analysis 
of coordinate descent on the dual problem. A key insight 
is that the iterates are obtained by reflecting and averaging, so they 
converge to an element of the dual feasible set that is minimal with 
respect to a ordering induced by the group associated with the penalty. 
Our main application is fused~lasso/total~variation denoising 
and isotonic regression on arbitrary graphs.  In those cases the associated 
group is a permutation group. \end{abstract}

\section{Introduction}
\label{sec:introduction}
Let $x \in \X$ be a vector of observations, and suppose we have a 
collection of procedures that we could potentially apply to the data. 
\begin{quote}
What functions of the data contain sufficient information for computing 
all of the procedures under consideration? 
\end{quote}
This is the fundamental question of computational sufficiency 
\citep{vu:group}, and this paper seeks answers to questions like these when 
the procedures are based on generalized lasso penalties. 

\subsection{Motivation}
Suppose the data are ordered and the underlying signal is suspected to be smooth. Then we might consider the fused lasso 
\citep{tibshirani.saunders.ea:sparsity}. 
This is also known as total variation denoising, 
and it is defined to be the solution of the penalized least squares problem: 
\begin{equation}
    \label{eq:1d_fused_lasso}
    \min_{\theta \in \Real^n} \frac{1}{2} \norm{x - \theta}^2 
    + \lambda \sum_{j=1}^{n-1} \abs{ \theta_{j+1} - \theta_j }
    \,.
\end{equation}
The tuning parameter $\lambda \geq 0$ controls the smoothness of the estimate, 
with larger $\lambda$ resulting in a more piecewise constant estimate.
This problem is relatively well-understood in terms of 
statistical theory \citep{mammen.:locally,davies.kovac:local,%
levy-leduc.harchaoui:catching,rinaldo:properties,rojas.wahlberg:on,%
lin.sharpnack.ea:sharp} and algorithms \citep{davies.kovac:local,%
tibshirani.taylor:solution,hofling:path,johnson:dynamic,condat:direct}. 
For example, nearly optimal theoretical error bounds have been derived 
and $O(n)$ time complexity algorithms are available.

Some might argue that the applicability of \Cref{eq:1d_fused_lasso} depends 
on the underlying distribution of the data.  For example, if the data were 
binary then it seems prudent to replace \Cref{eq:1d_fused_lasso} by the penalized logistic problem, 
\begin{equation}
    \label{eq:generalized_fused_lasso}
    \min_{\theta \in \Real^p}
    \braces[\bigg]{ \sum_{j=1}^n f(x_j) - x_j \theta_j } 
    + \lambda \sum_{j=1}^{n-1} \abs{ \theta_{j+1} - \theta_j }\,,
\end{equation}
with $f(t) = \log(1 + \exp(t))$.  This is a seemingly more difficult 
optimization problem, but surprisingly, 
\citet[section 3.2]{dumbgen.kovac:extensions} discovered that solutions to 
problems of the form \Cref{eq:generalized_fused_lasso}, with an exponential 
family log-likelihood, could be obtained by a simple transformation of the 
penalized least squares solution \Cref{eq:1d_fused_lasso}.

In another direction, if the data were arranged on a 2d grid---like an image---then we might replace the first differences in 
\Cref{eq:1d_fused_lasso} by differences across adjacent nodes of the grid. 
\citet{tibshirani.taylor:solution} have pointed 
out that this and \Cref{eq:1d_fused_lasso} are instances of a 
\emph{generalized lasso} problem,
\begin{equation}
    \label{eq:generalized_lasso}
    \min_\theta \frac{1}{2} \norm{x - \theta}^2 
    + \lambda \norm{D \theta}_1 \,,
\end{equation}
where $D$ is a $m \times n$ matrix that encodes desired structural and 
sparsity properties of the estimate, and $\norm{}_1$ is the $\ell_1$ norm which sums the absolute values of the entries of a vector.
The 2d variant replaces $D$ by a first difference operator on a grid and is 
known as 2d total variation denoising \citep{rudin.osher.ea:nonlinear}. More 
generally, this formulation supports arbitrary graphs and other penalties beside fused 
lasso.  Finally, we could generalize both 
\Cref{eq:generalized_fused_lasso} and \Cref{eq:generalized_lasso} to obtain a doubly generalized lasso
\begin{equation}
    \label{eq:doubly_generalized_lasso}
    \min_\theta \phi(\theta) - \innerp{x}{\theta} 
    + \lambda \norm{D \theta}_1 \,,
\end{equation}
where $\phi : \Real^n \to \Real$ is a ``nice enough'' convex function. 
In progressing from \Cref{eq:1d_fused_lasso} to 
\Cref{eq:doubly_generalized_lasso} 
we seem to be increasing the flexibility of our methods at the cost of 
potentially more challenging computation. This paper will show that this 
is not necessarily the case.

\subsection{What this paper is really about}

Since \Cref{eq:generalized_fused_lasso} can be reduced to 
\Cref{eq:1d_fused_lasso}, we might as well concentrate our efforts on 
efficient solvers for the least squares case. 
This is an algorithmic perspective, but we 
could also think about the reduction from an inferential point of view. 
The sparsifying effect of the $\ell_1$ penalty induces piecewise constant 
solutions of \Cref{eq:1d_fused_lasso}.  The location of change points between 
the pieces are often an inferential target 
\citep{levy-leduc.harchaoui:catching}. If the transformation 
from \Cref{eq:1d_fused_lasso} to \Cref{eq:generalized_fused_lasso} is 
smooth, then no new change points can be introduced by the 
transformation, so \Cref{eq:generalized_fused_lasso} cannot have any 
more power than \Cref{eq:1d_fused_lasso}.

Both the algorithmic and inferential points of view are interesting to us, 
but the underlying result that enables those interpretations is the 
discovery by \citet{dumbgen.kovac:extensions} of a relationship between the 
least squares fused lasso and generalized fused lasso problems.  It turns out that there is a deeper reason behind this phenomenon that will allow us to 
relate the least squares generalized lasso \Cref{eq:generalized_lasso} 
to the doubly generalized lasso \Cref{eq:doubly_generalized_lasso}, 
and similarly for extensions to other methods such as isotonic regression 
\citep{barlow.bartholomew.ea:statistical} that are not strictly based on generalized lasso penalties.

Here is an informal selection of our results. There is a group of orthogonal transformations associated with each specific generalized lasso 
penalty. Solutions of the 
doubly generalized lasso problem \Cref{eq:doubly_generalized_lasso} can be 
obtained by a simple transformation of the solution of 
\Cref{eq:generalized_lasso} whenever $\phi$ is 
invariant under the action of that group. When the penalty is the fused lasso on a graph, that group is simply a group of permutations restricted to the connected components of the graph. 
This immediately recovers the aforementioned result of 
\citet{dumbgen.kovac:extensions}. 
In the language of computational sufficiency \citep{vu:group}, we can say that 
\Cref{eq:generalized_lasso} is \emph{computationally sufficient} for 
\Cref{eq:doubly_generalized_lasso}.  Moreover, if we consider an entire class 
of procedures of the form \Cref{eq:doubly_generalized_lasso},  
sharing the same penalty and group of invariances, 
then penalized least squares \Cref{eq:generalized_lasso} is a member of 
that class and hence \emph{computationally minimal}.

The paper is organized as follows. 
We begin in \Cref{sec:preliminaries} by
reviewing concepts from \citet{vu:group} that are used throughout the 
paper. \Cref{sec:g_minimality} introduces the notion of a group minimal 
element and develops a general theory of computational sufficiency and 
minimality based on finding a group minimal element for a certain dual problem. The 
remainder of the paper specializes this theory to estimators based on so-called solar penalties. These penalties and their connection with 
generalized lasso penalties and reflection groups are discussed in \Cref{sec:solar_penalties}. 
The main theoretical result on these penalties is presented in 
\Cref{sec:polygonal_path} where we prove the existence of 
group minimal elements by analyzing a dual coordinate descent algorithm. This 
allows the computational sufficiency theory to be applied 
to estimators based on solar penalties. 
\Cref{sec:applications} discusses the consequences for specific examples 
including lasso, fused lasso, isotonic regression, and trend filtering.
Finally, \Cref{sec:discussion} discusses our results and directions forward. 
\section{Preliminaries}
\label{sec:preliminaries}
We begin by reviewing some concepts from \citet{vu:group} and 
introducing notation that will be used throughout the paper.

\subsection{Computational sufficiency}
Let $\X$ be a Euclidean space with the usual inner product $\innerp{}{}$ and norm $\norm{}$. 
The framework of computational sufficiency views procedures as 
set-valued functions on $\X$.  The basic idea is very simple. We wish 
to find functions of the data that contain sufficient information 
for computing every procedure in a collection.
\begin{definition}
Let $\mathcal{M}$ be a collection of set-valued 
functions on $\X$. A function $R$ on $\mathcal{X}$ is 
\emph{computationally sufficient} for $\mathcal{M}$ if for each  
$T \in \mathcal{M}$, there exists a set-valued function $f$ such that 
$f(T, R(x)) \neq \emptyset$ whenever $T(x) \neq \emptyset$ and 
\begin{equation*}
    f(T, R(x)) \subseteq T(x)
    \quad \text{for all} \quad x \in \X
    \,.
\end{equation*}
A function $U$ on $\X$ is \emph{computationally necessary} for $\mathcal{M}$  
if for each $R$ that is computationally sufficient for $\mathcal{M}$, 
there exists $h$ such that 
\begin{equation*}
    U(x) = h(R(x)) \quad \text{for all} \quad x \in \mathcal{X} 
    \,.
\end{equation*}
If $U$ is computationally necessary and sufficient, 
then $U$ is \emph{computationally minimal}.
\end{definition}
An easy way to establish computational minimalilty is to check that 
a sufficient reduction belongs to the collection under consideration.
\begin{lemma}[\citet{vu:group}]
\label{lem:trivially_minimal}
If $R \in \mathcal{M}$ is singleton-valued and 
computationally sufficient for $\mathcal{M}$, then 
$R$ is computationally minimal.
\end{lemma}

\subsection{Expofam-type estimators}
The procedures that we study in this paper are generalizations of penalized 
maximum likelihood for exponential family models. Let $C \subseteq \X$ be a 
nonempty set. The \emph{support function} of $C$ is the function 
$h_C : \X \to \Real \cup \set{+\infty}$ with values
\begin{equation*}
    h_C(x) = \sup_{z \in C} \innerp{z}{x} \,.
\end{equation*}
Note that $h_C = h_{\conv(C)}$ where $\conv(C)$ is the closed convex 
hull of $C$ \citep[Proposition 7.13]{bauschke.combettes:convex}, so the 
support function ``sees'' only the convex hull. 
The functions that we work with in this paper will generally be extended 
real valued. The set of  proper (finite for at least one point) closed (lower semicontinuous) convex functions $f :\X \to \Real \cup \set{+\infty}$ is denoted by $\Gamma_0(\X)$.
\begin{definition}
A set-valued function $T$ on $\X$ is an \emph{expofam-type} estimator if it 
has the form 
\begin{equation}
    \label{eq:expofam-type}
    T(x) = \argmin_\theta \phi(\theta) - \innerp{x}{\theta} + h_C(\theta) \,,
\end{equation}
where $\phi \in \Gamma_0(\X)$ is called the \emph{generator} of $T$, 
and $C \subseteq \X$ is a nonempty closed convex set called the 
\emph{penalty support set} of $T$. 
\end{definition}
Since \Cref{eq:expofam-type} is a generalization of penalized maximum 
likelihood for exponential families, many popular penalized estimators can 
be viewed as expofam-type estimators.  For example, when $x \in \Real^n$, 
isotonic regression and least squares fused lasso can both be expressed in 
the form
\begin{equation}
    \label{eq:iid_generator}
    \braces[\bigg]{ \sum_{i=1}^n f(\theta_i) - x_i \theta_i } + h_C(\theta)\,,
\end{equation}
where $f \in \Gamma_0(\Real)$. For isotonic regression, $h_C$ is the convex 
indicator of the cone of monotone vectors; for fused lasso, $h_C$ is 
proportional to the $\ell_1$ norm of the first differences.  We will describe 
these examples in more detail in \Cref{sec:solar_penalties}.  
See \citet{vu:group} for other examples.

\subsection{Group majorization}
\citet{vu:group} observed that the generators of expofam-type 
estimators often exhibit symmetries.  Note that the generator in 
\Cref{eq:iid_generator}\,,
\begin{equation*}
    \phi(\theta) = \sum_{i=1}^n f(\theta_i) \,,
\end{equation*}
is invariant under permutations of $\theta_1,\ldots,\theta_n$.  
This type of symmetry appears when $T$ corresponds to an i.i.d.\ statistical 
model. For example, $f$ may be the log-partition function of a one-dimensional 
exponential family.  
We describe symmetries like these in terms of a group of transformations. 
Let $\OrthoGroup(\X)$ be the orthogonal group of 
$\X$. In this paper $\G \subseteq \OrthoGroup(\X)$ 
will generically denote a compact group acting linearly on $\X$ and we denote the action of 
$g \in \G$ on $x \in \X$ by $g \cdot x$. A function $\phi$ is $\G$-invariant 
if $\phi(g \cdot x) = \phi(x)$ for all $x \in \X$ and $g \in \G$. 
A set $K \subseteq \X$ is $\G$-invariant if $g \cdot K = K$ for all $g \in \G$. 
The orbit of $x$ under $\G$ is the set
\begin{equation*}
    \G \cdot x = \set{g \cdot x \given g \in \G} \,.
\end{equation*}
The convex hull of the orbit, $\conv(G \cdot x)$, 
is the orbitope of $\G$ with respect to $x$. The group induces 
a preorder (reflexive and transitive) on $\X$ via inclusion of its 
orbitopes.
\begin{definition}
Let $x, y \in \X$. We say that $x$ is \emph{$\G$-majorized} by $y$, 
denoted by $x \preceq_\G y$, if $x \in \conv(G \cdot y)$.
\end{definition}
When $\G$ is the permutation group $\PermGroup_n$ acting 
on $\Real^n$, the ordering $\preceq_\G$ is exactly the (classical) 
majorization ordering \citep[see][]{marshall.olkin.ea:inequalities}.
$\G$-majorization was developed as an extension to more general subgroups of 
the orthogonal group and studied in depth by \citet{eaton.perlman:reflection}. 
One important discovery from that work is that if $\G$ is a finite reflection 
group (defined in \Cref{sec:reflection_groups}), then $\preceq_\G$ is a
cone ordering on the fundamental domain of the group. In the case of the 
permutation group, this phenomenon is realized by the monotone rearrangement 
definition of majorization. See \citet{giovagnoli.wynn:g-majorization,eaton:on,steerneman:g-majorization,francis.wynn:subgroup} for additional developments. 
Chapter 14.C of \citet{marshall.olkin.ea:inequalities} gives an overview of 
the connections between classical majorization and $\G$-majorization. 
Our main use of $\G$-majorization is in its application to convex optimization 
via the following lemma.

\begin{lemma}[\cite{giovagnoli.wynn:g-majorization}]
\label{lem:g_majorization_equivalence}
\label{eq:g_monotone}
The follow statements are equivalent:
\begin{enumerate}[label=(\alph*),ref=(\alph*)]
    \item \label{case:g_majorization}
    $x \preceq_\G y$,
    \item \label{case:g_majorization_orbit_inclusion}
    $\conv(G \cdot x) \subseteq \conv(G \cdot y)$,
    \item \label{case:g_majorization_support_function}
    $h_{\G \cdot x}(u) \leq h_{\G \cdot y}(u)$ for all $u \in \X$,
    \item \label{case:g_majorization_function_minimization}
    $f(x) \leq f(y)$ for all $\G$-invariant convex $f$.
\end{enumerate}
\end{lemma}
The last statement in the above lemma is a generalization of Schur convexity 
to general groups. We will make extensive use of this $\G$-monotonicity condition throughout the paper. 
\section{Group minimality to computational minimality}
\label{sec:g_minimality}
The $\G$-monotonicity condition in \Cref{eq:g_monotone} suggests that 
it may be possible to universally optimize families of $\G$-invariant convex 
functions by finding minimal elements in the $\G$-majorization ordering. 
In this section we will show how this basic observation leads to a 
computationally minimal reduction for expofam-type estimators with 
$\G$-invariant generators.

\subsection{Convex duality and G-minimality}
Using standard arguments from convex analysis 
\citep[see][Chapter 15.2, Example 13.3]{bauschke.combettes:convex}, 
the Fenchel dual problem to \Cref{eq:expofam-type} can be written as 
\begin{equation}
  \label{eq:dual}
  \begin{aligned}
      &\text{minimize}
        &   &
        \phi^*(y)
      \\
      &\text{subject to}
        &   & y \in x - C
        \,,
  \end{aligned}
\end{equation}
where $\phi^*$ is the convex conjugate of $\phi$:
\begin{equation*}
    \phi^*(u) = \sup_\theta \innerp{u}{\theta} - \phi(\theta) \,.
\end{equation*}
Since $\phi$ is a proper closed convex function, 
its conjugate is also a proper closed convex function 
\citep[Corollary 13.38]{bauschke.combettes:convex}. 
We call \Cref{eq:expofam-type} the primal problem, and \Cref{eq:dual} 
the dual problem. Let us focus on the dual problem for now. 
One way to interpret the dual problem is to think about regression. The variable $y$ in \Cref{eq:dual} should be thought of as 
the fitted value, while the set $C$ represents constraints on the residual. 
Since $$(\text{data}) = (\text{fitted value}) + (\text{residual})\,,$$
we can think of the dual problem as subtracting residuals from $x$ to 
get the fitted value $y$, subject to a constraint on the residual.

Note that the dual feasible set, $x - C$, 
is independent of $\phi$---it is the dual counterpart of the  
penalty $h_C$. If $\phi$ is $\G$-invariant, then so is $\phi^*$:
\begin{align*}
    \phi^*(g \cdot u)
    &= \sup_\theta \innerp{g \cdot u}{\theta} - \phi(\theta) \\
    &= \sup_\theta 
        \innerp{u}{g^{-1} \cdot \theta} - \phi(g^{-1} \cdot \theta) 
     = \phi^*(u)\,.
\end{align*}
If we assume that $\phi$ is $\G$-invariant, then \Cref{eq:dual} involves 
minimizing a convex $\G$-invariant function over the dual feasible set 
$x - C$, so the ordering induced by $\G$ on $x - C$ may play an important 
role. This motivates the next definition.
\begin{definition}
An element $y \in K \subseteq \X$ is \emph{$\G$-minimal} 
in $K$ if $y \preceq_\G z$ for all $z \in K$.
\end{definition}
$\G$-minimality is closely tied with minimization over $K$.  
This is embodied in the following extension of 
\Cref{lem:g_majorization_equivalence}.
\begin{lemma}
\label{lem:g_minimal_equivalance}
Let $K \subseteq \X$ be a closed convex set and $y \in K$.
The following statements are equivalent:
\begin{enumerate}[label=(\alph*),ref=(\alph*)]
    \item \label{case:g_minimal}
    $y \preceq_\G z$ for all $z \in K$,
    \item \label{case:g_minimizes_convex_functions}
    $f(y) = \inf_{z \in K} f(z)$ for all $\G$-invariant convex $f$.
\end{enumerate}
\end{lemma}
In general, a $\G$-minimal element 
may not exist in a given set $K$, however existence becomes \emph{easier} 
as the size of the group increases, and there is actually only one possible 
candidate for a $\G$-minimal element.
\begin{lemma}
\label{lem:g_minimality_criteria}
Let $K$ be a closed convex set.
\begin{enumerate}[label=(\alph*),ref=(\alph*)]
\item \label{statement:g_minimality_trivial} 
If $\G$ is the trivial group, then $K$ has a $\G$-minimal element if and 
only if $K$ is a singleton.
\item \label{statement:g_minimality_increasing}
If $\mathcal{H} \subseteq \mathcal{G}$ and $y$ is $\H$-minimal in $K$, 
then $y$ is also $\G$-minimal in $K$.
\item \label{statement:g_minimality_orthogonal}
$K$ has a unique $\OrthoGroup(\X)$-minimal element.
\item \label{statement:g_minimality_minimum_norm}
If $y \in K$ is $\G$-minimal in $K$, then it is unique and equal to the 
minimum norm element of $K$.
\end{enumerate}
\end{lemma}
The proofs of \Cref{lem:g_minimal_equivalance,lem:g_minimality_criteria} 
are given in \Cref{sec:proofs}.
\begin{theorem}
\label{thm:g_minimality_and_minimal_norm}
A nonempty closed convex set $K \subseteq X$ has an element that is 
$\G$-minimal in $K$ if and only if
\begin{equation*}
    \argmin_{y \in K} \norm{y} \preceq_\G z
\end{equation*}
for all $z \in K$, and when this holds the minimum norm element is 
$\G$-minimal.
\end{theorem}
\Cref{thm:g_minimality_and_minimal_norm} is simply a restatement of 
\ref{statement:g_minimality_minimum_norm} in 
\Cref{lem:g_minimality_criteria} together with the definition of 
$\G$-minimality. The theorem is important for two reasons.  Firstly, it 
reduces the problem of finding a $\G$-minimal element to checking that the 
minimum norm element is $\G$-minimal. Secondly, it opens up different 
perspective on $\G$-minimality. Rather than thinking about $\G$-minimality 
in terms of a fixed group $\G$, it is more fruitful to determine the 
\emph{smallest group} such that the minimum norm element of $K$ 
is $\G$-minimal.

\subsection{The least squares reduction}
Now let us return to the dual problem \Cref{eq:dual} and suppose that $y_*$ is 
$\G$-minimal in $x - C$.  By \Cref{lem:g_minimal_equivalance}, $y_*$ is a dual 
solution for \emph{every} dual problem \Cref{eq:dual} where $\phi$ is 
$\G$-invariant.  In order to relate the dual solution to the primal solution, 
we need strong duality to hold. We shall assume this is the case in a generic 
way.
\begin{assumption}[Strong Duality]
    \label{assumption:strong_duality}
    The generator $\phi$ and penalty $h_C$ satisfy sufficient conditions to ensure that strong duality holds.
\end{assumption}
\citet[Chapters 15 and 19.1]{bauschke.combettes:convex} present a variety of 
sufficient conditions for \Cref{assumption:strong_duality} to hold. One 
particularly simple one is that $C$ is compact and $\phi$ is finite 
on some subset of $C$. Another is that
\begin{equation*}
    \relint(\effdom \phi) \cap \relint(\effdom h_C) \neq \emptyset \,,
\end{equation*}
where $\effdom$ is the effective domain (set of points where the function is finite) and $\relint$ is the relative interior (interior relative to the 
affine hull). When strong duality holds, the set of primal solutions can be recovered from an arbitrary dual solution $y$ via 
\begin{equation*}
    \argmin_\theta \phi(\theta) - \innerp{y}{\theta}
\end{equation*}
\citep[see][Corollary 19.2]{bauschke.combettes:convex}. So for a fixed 
$x$, if $x -C$ has a $\G$-minimal element, then we can solve the primal 
problem by finding the $\G$-minimal element of the dual feasible set. 
The latter does not depend on $\phi$ as long as $\phi$ is $\G$-invariant.

In order for the procedure we just discussed to be applicable to the problem 
of computational sufficiency, we to be able to find a $\G$-minimal in $x - C$ 
for each $x \in \X$. Since the \emph{only} candidate for a $\G$-minimal 
element is the minimum norm element, we can continue to develop our theory by 
assuming that it is indeed $\G$-minimal, and we can find it by least squares.  This turns out to be computationally minimal.
\begin{theorem}
\label{thm:least_squares_minimal}
Let $C \subseteq \X$ be a closed convex set, and suppose that $x - C$ 
contains a $\G$-minimal element for each $x \in \X$. If $\mathcal{M}$ is 
a collection of expofam-type estimators with penalty support set $C$ and   
$\G$-invariant generators $\phi \in \Gamma_0(\X)$ satisfying 
\Cref{assumption:strong_duality}, then 
\begin{equation*}
    U(x) = \argmin_{y \in x - C} \norm{y} 
\end{equation*}
is computationally sufficient for $\mathcal{M}$.  
In particular, if $T \in \mathcal{M}$ has generator $\phi$, then 
\begin{equation}
    \label{eq:primal_recovery}
    T(x) 
    = \argmin_\theta \phi(\theta) - \innerp{U(x)}{\theta} 
    \quad
    \text{for all}
    \quad
    x \in \X\,.
\end{equation}
Moreover, $U$ is in $\mathcal{M}$, is equal to the penalized least squares 
estimator,
\begin{equation*}
    U(x) = \argmin_{\theta} \frac{1}{2} \norm{x - \theta}^2 + h_C(\theta) \,,
\end{equation*}
and is computationally minimal for $\mathcal{M}$.
\end{theorem}
The proof is given in \Cref{sec:proofs}. 
\begin{remark}
It is basic result of convex analysis that the solution of 
\Cref{eq:primal_recovery} is given by $T(x) = \nabla \phi^*(U(x))$. 
When $\phi$ is a Legendre function \citep[Exercise 18.7]{bauschke.combettes:convex}, as is the case when $\phi$ is the log-partition 
function of a regular exponential family, 
$\nabla \phi^* = (\nabla \phi)^{-1}$.  This corresponds to a fundamental 
property of exponential families whereby the maximum likelihood estimator 
is a method of moments estimator.
\end{remark}
It is interesting to consider the meaning of \Cref{thm:least_squares_minimal} 
when $\phi$ is the log-partition function of an exponential family of 
distributions.  \Cref{eq:primal_recovery} is nothing other than the maximum 
likelihood estimator based on the statistic $U(x)$.  So the theorem says that 
when the dual feasible set admits a $\G$-minimal element, the penalized 
maximum likelihood estimators reduce to maximum likelihood with the data 
replaced by a penalized least squares fit, or equivalently, the residual from 
projecting the data onto $C$.
 
\section{Solar penalties and reflections}
\label{sec:solar_penalties}
Computational minimality of the least squares reduction in 
\Cref{thm:least_squares_minimal} depends on the existence of a $\G$-minimal 
element in the dual feasible set $x - C$ for each $x \in \X$. We can interpret 
$\G$-minimality of $y_*$ in $x - C$ as the requirement that $y_*$ be 
\emph{reachable} from any point in the dual feasible set by averaging a 
sequence of transformations restricted to $\G$. Since the only possible 
candidate is the minimum norm element, we should 
focus our efforts on finding a small group $\G$ that allows the minimum norm 
element to be \emph{reachable} for every $x$. This clearly depends on the 
geometry of $C$.  In this section we begin to explore these ideas by 
introducing a family of penalties encompassing the generalized lasso 
penalties. The key insight is that there is a natural group associated with 
each penalty that encodes the geometry of the penalty support set $C$.

\begin{definition}
A penalty is said to be a \emph{solar} penalty if it is the support function of a Minkowski sum of line segments and rays.\footnote{\emph{Solar} is an acronym for \underline{s}um \underline{o}f \underline{l}ine segments \underline{a}nd \underline{r}ays.}
\end{definition}
The generalized lasso penalty \citep{tibshirani.taylor:solution} is a special 
case of a solar penalty. Usually, it is written in matrix--vector form as
\begin{equation*}
    \lambda \norm{D \theta}_1 
    \,,
\end{equation*}
where $\lambda$ is a tuning parameter and $D$ is a penalty matrix. The main 
idea is to encourage sparsity of $D \theta$, and the matrix $D$ is chosen so 
that sparsity of $D\theta$ reflects some desired structure in the estimate.
The generalized lasso penalty can also be expressed as the support 
function of the image of the hypercube under the linear transformation $D^T$:
\begin{equation*}
    C = \set{ D^T z \given \norm{z}_\infty \leq \lambda }
    \,.
\end{equation*}
Such a set is sometimes called a \emph{zonotope}, which is a  
Minkowski sum of line segments \citep[Chapter 7]{ziegler:lectures}. Solar 
penalties, on the other hand, also allow the addition of rays. So the 
resulting set can possibly be unbounded. This allows hard constraints to be 
encoded into the penalty.

\subsection{Base representation of a solar penalty}
\begin{figure}[t]
\centering
\includegraphics[width=0.4\columnwidth]{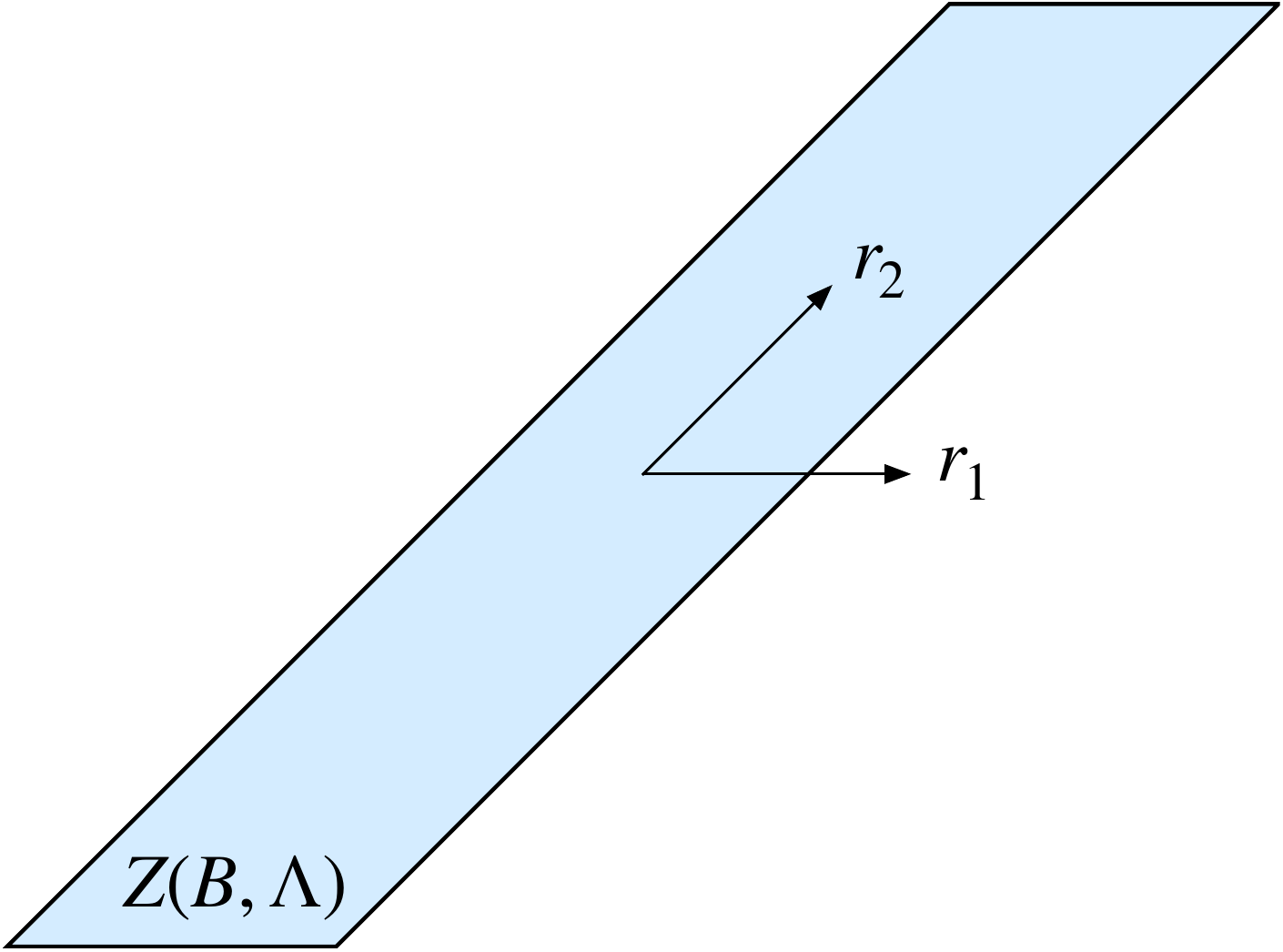}
\caption{
Solar penalty support set $Z(B,\Lambda)$ and its base vectors ($r_1,r_2$).}
\label{fig:solar_penalty}
\end{figure}
The \emph{base} representation of a solar penalty is 
a pair $(B, \Lambda)$ consisting of a tuple of unit vectors, 
$B = \parens{ r_1,\ldots,r_m }$ which we call \emph{bases} or base vectors, 
and a Cartesian product of closed intervals 
$\Lambda = I_1 \times \cdots \times I_m$, 
where $I_j \subseteq [-\infty,\infty]$, $j=1,\ldots,m$, can possibly be 
unbounded.  When convenient we may think of $B$ as being a matrix with 
unit norm columns. Then the solar penalty support set is
\begin{equation*}
    Z(B, \Lambda) 
    \coloneqq
    B \cdot \Lambda
    = 
    \set[\bigg]{ z \in \X 
    \given z = \sum_{j=1}^m \lambda_j r_j ,\; \lambda_j \in I_j }
    \,.
\end{equation*}
See \Cref{fig:solar_penalty} for an example.
The base representation of a solar penalty is useful, because it will allow us 
to connect $Z(B,\Lambda)$ to a reflection group 
(to be defined in \Cref{sec:reflection_groups}) that encodes some of the 
geometry and symmetry of the penalty.  Let us first discuss some examples 
for the case $\X = \Real^n$.

\begin{example}[Lasso]
Let $e_j$ denote the $j$th standard basis vector which has 
a $1$ in its $j$th coordinate and 0 elsewhere. 
The $\ell_1$ norm penalty has base representation
\begin{align*}
    B &= \parens[\big]{ e_1, \ldots, e_n } \\
    \Lambda &= [-\lambda,\lambda]^n \,,
\end{align*}
with $\lambda > 0$. Then 
\begin{equation*}
    Z(B,\Lambda) = \set{ z \in \Real^n \given \norm{z}_\infty \leq \lambda }
\end{equation*}
and $h_{Z(B,\Lambda)} = \norm{}_1$.
\end{example}

\begin{example}[Nonnegative regression]
If we replace the intervals in the base representation of the $\ell_1$ penalty 
by the nonpositive half of the real line, then 
\begin{align*}
    B &= \parens[\big]{ e_1, \ldots, e_n } \\
    \Lambda &= [-\infty, 0]^n\,,
\end{align*}
and $Z(B,\Lambda)$ becomes the convex cone of vectors with nonpositive entries.  The corresponding support function is the convex indicator of the 
nonnegative cone,
\begin{equation*}
    h_{Z(B,\Lambda)}(\theta) = 
    \begin{cases}
        0 & \text{if $\theta_j \geq 0$ for all $j$,} \\
        +\infty & \text{otherwise.}
    \end{cases}
\end{equation*}
This is an example of a hard constraint. We could relax the constraint by replacing the intervals by
\begin{equation*}
    \Lambda = [-\lambda, 0]^n \,.
\end{equation*}
In this case, the penalty operates like an $\ell_1$ norm on negative entries:
\begin{equation*}
    h_{Z(B,\Lambda)}(\theta) = \lambda \sum_j \abs{\min(0, \theta_j)} \,.
\end{equation*}
\end{example}

\begin{example}[Fused lasso]
The 1d fused lasso is the $\ell_1$ norm of the first differences. 
This penalty is often used 
for smoothing and signal approximation when the data follow a one-dimensional 
structure. It encourages estimates that are 
piecewise constant. A base representation of this penalty is given by
\begin{align*}
    B &= \parens[\big]{ (e_{j+1} - e_j) / \sqrt{2} \given j=1,\ldots,n-1} \\
    \Lambda &= [-\lambda,\lambda]^{n-1}
\end{align*}
Then
\begin{equation*}
    h_{Z(B,\Lambda)}(\theta) 
    = \frac{\lambda}{\sqrt{2}} \sum_{j=1}^{n-1} \abs{\theta_{j+1} - \theta_j}
    \,.
\end{equation*}
In the engineering and signal processing literature the penalty 
is also known as \emph{total variation} and more often employed in the case of 2d signals such as images \citep{rudin.osher.ea:nonlinear}.  Both the 1d and 
2d cases can be described succinctly by introducing an undirected graph 
on $\bracks{p}$ with edge set $\mathcal{E} \subseteq \bracks{p}^2$.
The 1d fused lasso correponds to a chain graph, while the 2d case corresponds 
to a 2d grid. The base representation then becomes
\begin{align*}
    B &= \parens[\big]{ (e_j - e_i) / \sqrt{2} \given \set{i,j} \in \mathcal{E} } \\
    \Lambda &= [-\lambda,\lambda]^{\card{\mathcal{E}}} \,.
\end{align*}
This graph-guided fused lasso penalty encourages estimates that are 
piecewise constant across nodes of the graph.
\end{example}
\begin{example}[Isotonic regression]
Isotonic regression applies to linearly ordered data when the goal is to 
produce a monotonic fit. There is a large literature on the subject and the 
book by \citet{barlow.bartholomew.ea:statistical} is a standard reference. 
As a solar penalty, the isotonicity constraint can be viewed as a mix of fused 
lasso with a nonnegativity constraint on the first differences.  The base 
representation of the penalty is
\begin{align*}
    B &= \parens[\big]{ (e_{j+1} - e_j) / \sqrt{2} \given j=1,\ldots,n-1} \\
    \Lambda &= [-\infty,0]^{n-1} \,,
\end{align*}
and the corresponding support function is the convex indicator function of the 
monotone cone,
\begin{equation*}
    h_{Z(B,\Lambda)}(\theta) 
    = 
    \begin{cases}
        0 &\text{if $\theta_{j+1} \geq \theta_j$ for all $j$,} \\
        +\infty & \text{otherwise.}
    \end{cases}
\end{equation*}
Nearly-isotonic regression \citep{tibshirani.hofling.ea:nearly-isotonic} 
relaxes the isotone constraint and can be viewed 
as an $\ell_1$ penalty on positive first differences. It is obtained by 
replacing the above intervals by ones of the form $[-\lambda, 0]$. 
Similarly to the fused lasso, the neighborhood structure can generalized from 
a chain graph to an arbitrary graph by making the obvious modification to the 
base. Though this seems to make the more sense with trees so that a partial 
order can be maintained.
\end{example}

\begin{example}[Trend filtering]
Our final example is $\ell_1$ trend filtering \citep{kim.koh.ea:1}. Like 
the fused lasso, trend filtering applies to data with an ordering structure, 
but it uses the second difference instead of the first difference. 
This encourages estimates that are piecewise linear. In the 
case of linearly ordered data, the base representation of the trend 
filtering penalty has the form, 
\begin{align*}
    B &= \parens[\big]{ (e_{j+2} - 2 e_{j+1} + e_j) / \sqrt{6} \given j=1,\ldots,n-2} \\
    \Lambda &= [-\lambda,\lambda]^{n-2}
    \,.
\end{align*}
The corresponding support function is an $\ell_1$ norm of the second 
differences:
\begin{equation*}
    h_{Z(B,\Lambda)}(\theta) 
    = \frac{\lambda}{\sqrt{6}} \sum_{j=1}^{n-2} 
        \abs{ (x_{j+2} - x_{j+1}) - (x_{j+1} - x_{j}) }
    \,.
\end{equation*}
\end{example}

We can form new solar penalties by taking the union of bases of existing 
penalties. This is effectively the same as adding the respective support 
functions. For example, the sparse fused lasso is the sum of the lasso and 
fused lasso penalties. Most importantly, the class of solar penalties is 
closed under addition with a separable support function.

\subsection{The reflection group associated with a solar penalty}
\label{sec:reflection_groups}
\begin{figure}[t]
\centering
\includegraphics[width=0.6\columnwidth]{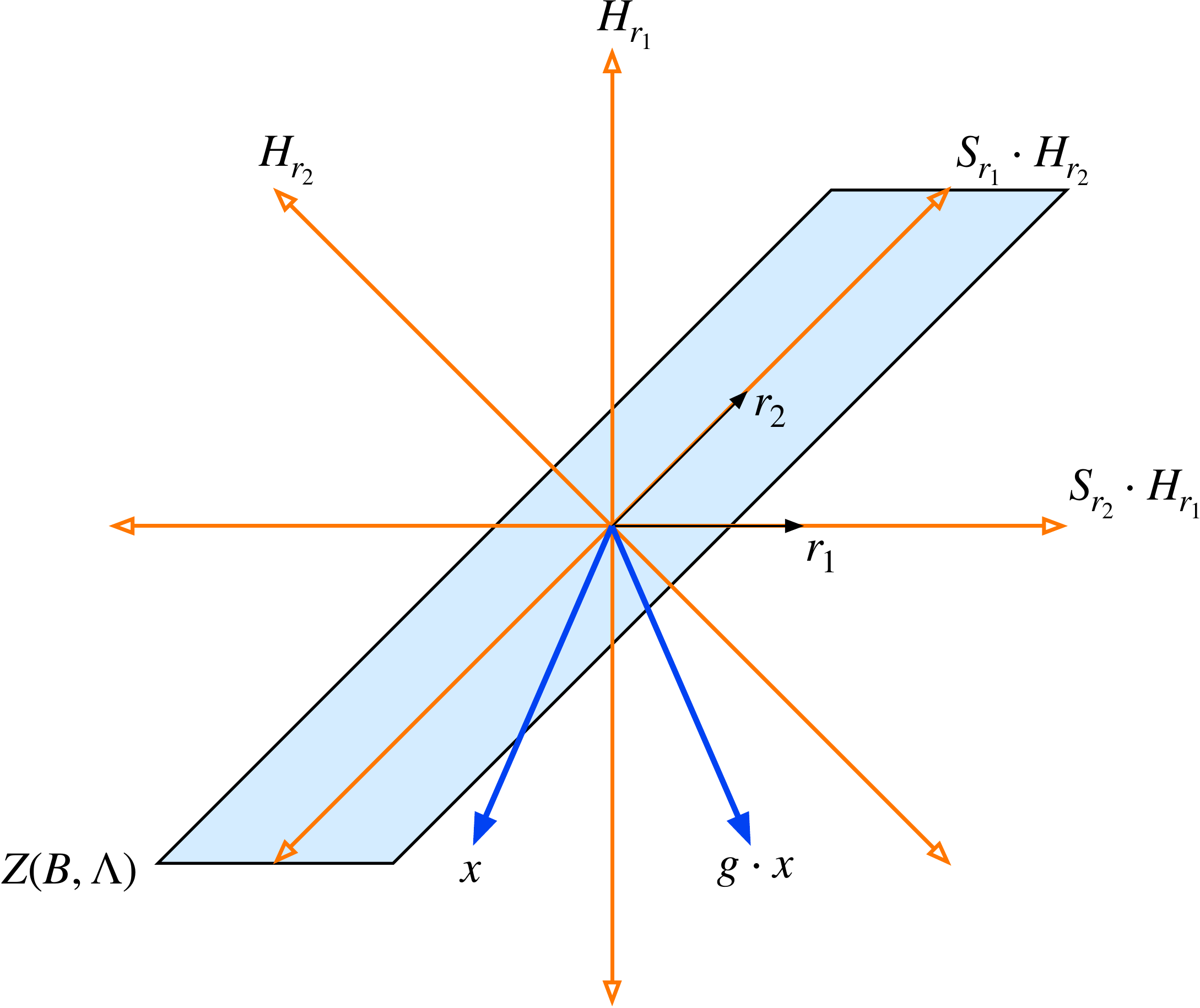}
\caption{
Reflection group $\G(B)$ associated with the solar penalty with base vectors ($r_1,r_2$).  The group is generated by two reflections, $S_{r_1}$ and 
$S_{r_2}$, across the hyperplanes $H_{r_1}$ and $H_{r_2}$.  The remaining two hyperplanes are $S_{r_1} \cdot H_{r_2}$ and $S_{r_2} \cdot H_{r_1}$. 
The four hyperplanes partition 
the space into 8 chambers (cones) that the group acts transitively on.
}
\label{fig:solar_reflection}
\end{figure}
The base vectors in the base representation of a solar penalty have a natural correspondence with hyperplanes in $\X$.  Let $r \in \X$ be a unit vector and 
\begin{equation*}
    H_r \coloneqq \set{ x \in \X \given \innerp{x}{r} = 0 } \,.
\end{equation*}
$H_r$ is the hyperplane normal to $r$. The linear transformation $S_r : \X \to \X$ defined by $S_r \cdot x = x - 2 r \innerp{r}{x}$ is called the 
\emph{reflection} across $H_r$ (or along $r$).  Note that it satisfies
\begin{align*}
    S_r \cdot r &= -r \\
    S_r \cdot H_r &= H_r \,.
\end{align*}
So it flips the sign of $r$ and leaves $H_r$ invariant. Associated to 
each solar penalty is a set of reflections $\set{S_r \given r \in B}$. 
These reflections generate a group of transformations.
\begin{definition}
The \emph{reflection group} generated by a set of base vectors $B$, 
denoted $\G(B)$, is the smallest closed subgroup of $\OrthoGroup(X)$ 
containing the set of reflections $\set{S_r \given r \in B}$.
\end{definition}
See \Cref{fig:solar_reflection} for an illustration. Note that although the 
group $\G(B)$ is finitely generated, it may be possible that $\G(B)$ is an 
infinite group.  We will see later that this is the case for trend filtering.
\Cref{tab:reflection_groups} lists the groups associated with each of the 
other example solar penalties discussed above. We will derive and discuss these in detail in \Cref{sec:applications}, but first we return to the problem of reachability of the minimum norm element.

\begin{table}[t]
\centering
\begin{tabular}{c|l|l}
    \hline
    Group & Action & Penalty \\
    \hline
    $(\Integer_2)^n$ & sign change 
    & lasso, nonnegative regression \\
    $\PermGroup_n$ & permutation 
    & fused lasso, isotonic regression \\
    $(\Integer_2)^n \rtimes \PermGroup_n$ & sign change and permutation 
    & sparse fused lasso 
\end{tabular}
\caption{Some reflection groups, their action on $\X$, and associated solar penalties.}
\label{tab:reflection_groups}
\end{table}
 
\section{Existence of a polygonal path}
\label{sec:polygonal_path}
To apply the general theory of \Cref{sec:g_minimality} to solar 
penalized estimators we need to establish the existence of a $\G$-minimal 
element in the dual feasible set $x - Z(B,\Lambda)$ for some appropriate 
group. \Cref{thm:g_minimality_and_minimal_norm} tells us that the only 
candidate for a $\G$-minimal element is the minimum norm element. 
The following theorem shows that for a solar penalty, the reflection group 
generated by its base is an appropriate choice.
\begin{theorem}
\label{thm:minimal_norm_solar}
Let $(B,\Lambda) = (\parens{r_1,\ldots,r_m}, I_1 \times \cdots \times I_m)$ 
be the base representation of a solar penalty and $x \in \X$.
If $\G \equiv \G(B)$ is the reflection group generated by $B$, then 
the minimum norm element of $x - Z(B,\Lambda)$ is $\G$-minimal in 
$x - Z(B,\Lambda)$.
\end{theorem}
The proof of \Cref{thm:minimal_norm_solar} is similar in spirit to the path 
result of \citet[page 47]{hardy.littlewood.ea:inequalities} \citep[see][Lemma B.1 in Chapter 2]{marshall.olkin.ea:inequalities} and its generalization 
by \citet{eaton.perlman:reflection}.  Let $y_*$ be the minimum norm element 
of $Z(B,\Lambda)$ and $y_0 \in x - Z(B, \Lambda)$. We will construct a 
polygonal path from $y_0$ to $y_*$ where each segment along the path is 
obtained by averaging elementary transformations of the iterates. 
Our approach to constructing this path is algorithmic and was inspired by 
the boundary lemma proof of \citet{tibshirani.taylor:solution}. 
Consider applying cyclic coordinate descent to the minimum norm problem:
\begin{equation}
  \label{eq:minimum_norm_coordinate}
  \begin{aligned}
      &\text{minimize}
        &   &
        \norm*{x - {\textstyle \sum_{j=1}^m} \alpha_j r_j}
      \\
      &\text{subject to}
        &   & \alpha_j \in I_j\,, \quad j = 1,\ldots,m
        \,.
  \end{aligned}
\end{equation}
Rather than tracking the iterates in terms of $(\alpha_1,\ldots,\alpha_m)$,
we track them in terms of the \emph{fitted values}, 
\begin{equation*}
  y_t = x - \sum_j \alpha_j^{(t)} r_j \,, \quad t=0,1,2, \ldots \,.
\end{equation*}
A coordinate descent  update modifies $y_t$ along one coordinate, say $j$, 
so that 
\begin{equation*}
  y_{t+1} - y_t 
  = \braces[\Big]{ \alpha_j^{(t+1)} - \alpha_j^{(t)} } r_j
  \in \vecspan (r_j) \,.
\end{equation*}
So the selection of coordinate can instead be viewed as the selection of a 
base vector $r \in B$, and the update is find a minimum norm element along 
a line segment parallel to $r$. This yields the following geometric 
description of coordinate descent applied to \Cref{eq:minimum_norm_coordinate}.
\begin{enumerate}
  \item Select the coordinate to be updated: $r \in B$.
  \item Let $y_{t+1}$ be the minimum norm element in 
  $(y_t + \vecspan(r)) \cap (x - Z(B,\Lambda))$.
\end{enumerate}
\begin{figure}[t]
\centering
\includegraphics[width=0.9\columnwidth]{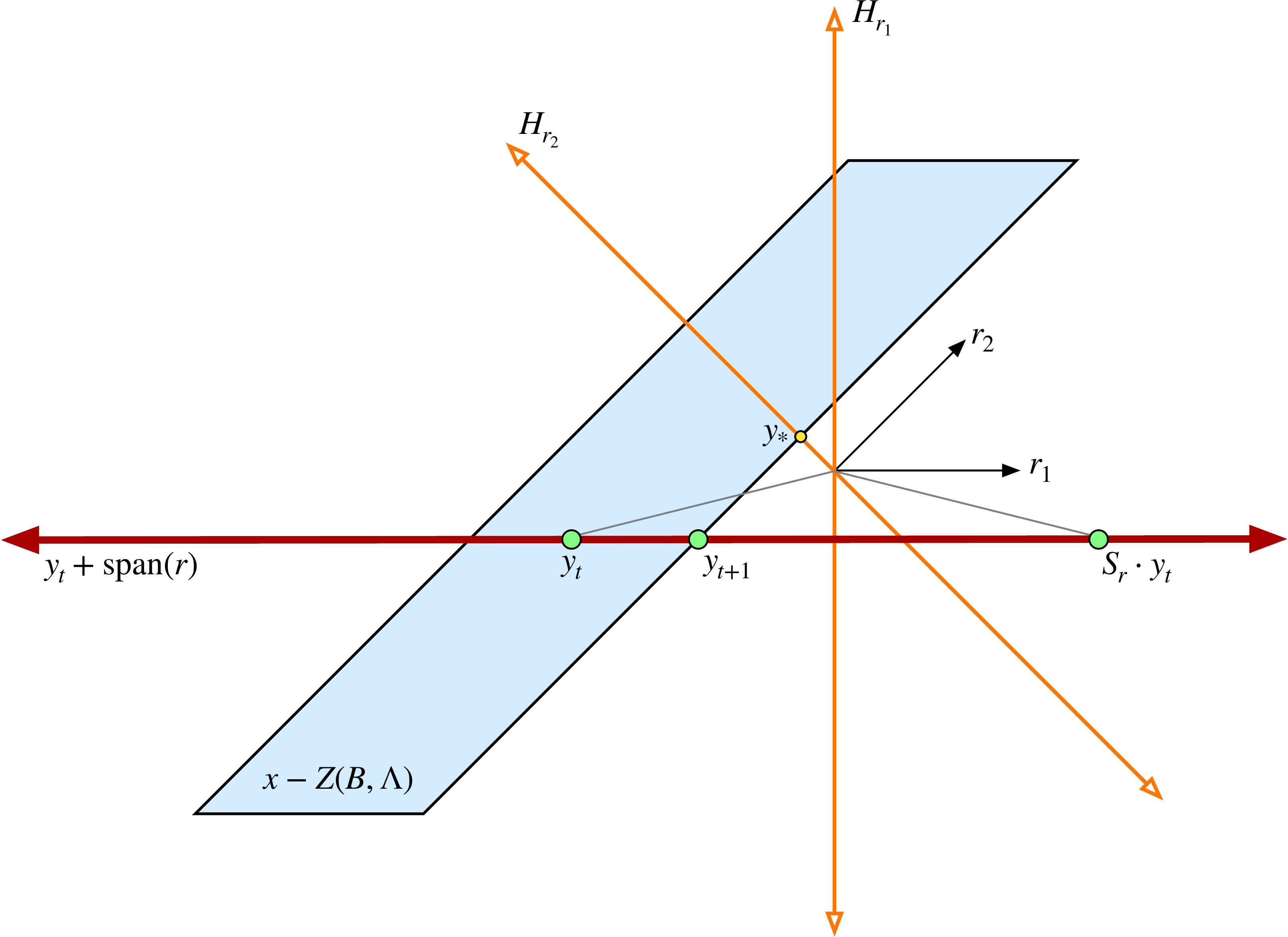}
\caption{
Coordinate descent update with $r = r_1$.  The update moves $y_t$ to 
$y_{t+1}$ on a line parallel to $r$. 
The minimum norm point in $(y_t + \vecspan(r)) \cap (x - Z(B,\Lambda)$ is 
obtained by averaging $y_t$ with its reflection $S_r \cdot y_t$ across the 
hyperplane $H_r$.
}
\label{fig:solar_coordinate_descent}
\end{figure}
\Cref{fig:solar_coordinate_descent} illustrates the update, and it provides a 
visual explanation for the following lemma which we will use to prove 
\Cref{thm:minimal_norm_solar}.
\begin{lemma}
\label{lem:ccd_g_monotone}
Let $y_0 \in x - Z(B, \Lambda)$ and $y_t$, $t=1,2,\ldots$, be the iterates of 
cyclic coordinate descent applied to the minimum norm problem 
\Cref{eq:minimum_norm_coordinate}. The iterates form a $\G$-monotone decreasing sequence with $\G \equiv G(B)$, 
\begin{equation*}
  y_t \succeq_\G y_{t+1} \quad \text{for all} \quad t\,,
\end{equation*}
and $y_t \to y_*$, the minimum norm element of $x - Z(B, \Lambda)$, 
as $t \to \infty$.
\end{lemma}
\begin{proof}[Proof of \Cref{lem:ccd_g_monotone}]
\citet[Theorem 4.1]{tseng:convergence} 
establishes the convergence of $y_t$ to $y_*$, because the constraints of 
\Cref{eq:minimum_norm_coordinate} are separable, the objective function is 
convex, and the lower level set 
$\set{y \in \X \given \norm{y} \leq \norm{y_0}}$ is compact. 
Moving on to the remaining claim, 
note that $y_t$, $S_r \cdot y_t$, and $y_{t+1}$ all lie on the line 
$y_t + \vecspan(r)$. Since 
\begin{equation*}
  \norm{y_{t+1}} \leq \norm{y_t} = \norm{S_r \cdot y_t} \,,
\end{equation*}
the line segments $[0, y_t]$, $[0, S_r \cdot y_t]$ form the legs of an 
isoceles triangle with $y_{t+1}$ in the base of the triangle. It then follows that 
\begin{equation*}
  y_{t+1}  \in  \conv\set{y_t, S_r \cdot y_t} 
\end{equation*}
and $y_{t+1} \preceq_\G y_t$ as desired.
\end{proof}
An interesting feature of this proof is that in the case of the permutation group, it shows that the coordinate descent update is an elementary Robin Hood operation \citep{arnold.sarabia:majorization}. More generally, the geometry 
of the coordinate descent updates is described by the reflection group. The 
iterates progress towards the minimum norm element by sequentially reflecting 
and averaging along each base vector, so the iterates must lie in 
$\conv(\G \cdot y_0)$.
\begin{proof}[Proof of \Cref{thm:minimal_norm_solar}]
By induction, $y_t \in \conv(G \cdot y_0)$. Since the orbitope is compact, 
$y_* \in \conv(G \cdot y_0)$ and hence $y_* \preceq_\G y_0$ for all 
$y_0 \in x - Z(B, \Lambda)$.
\end{proof}
 
\section{Consequences}
\label{sec:applications}
The immediate consequence of \Cref{thm:minimal_norm_solar} is that the 
minimal norm element of the dual feasible set for the solar penalty is 
$\G$-minimal.  Combining with \Cref{thm:least_squares_minimal} 
yields the main result of the paper:
\begin{theorem}
\label{thm:solar_minimality}
Let $h_C$ be a solar penalty with base representation $(B,\Lambda)$ 
and $\G \equiv \G(B)$ be the reflection group generated by $B$.
If $\mathcal{M}$ is a collection of expofam-type estimators with 
solar penalty $h_C$ and $\G$-invariant generators 
$\phi \in \Gamma_0(\X)$ satisfying \Cref{assumption:strong_duality}, then 
the penalized least squares estimator in $\mathcal{M}$ is computationally 
minimal for $\mathcal{M}$.
\end{theorem}
In the remainder of this section we will apply 
\Cref{thm:solar_minimality} in more detail 
to different families of estimators by taking the following steps.
\begin{enumerate}
    \item Fix a solar penalty with base representation $(B,\Lambda)$.
    \item Let $\G = \G(B)$ be the reflection group generated by $B$.
    \item Let $\mathcal{M}$ be a collection of expofam-estimators with 
    $\G$-invariant generators satisfying \Cref{assumption:strong_duality}.
\end{enumerate}
We will follow this program for each of the example solar penalties in \Cref{sec:solar_penalties} with $\X = \Real^n$.  The main challenge is 
identifying the reflection group.

\subsection{Lasso and nonnegative regression}
\label{sec:lasso}
The base vectors of the $\ell_1$ norm are 
\begin{equation*}
    B = \parens{e_1,\ldots,e_n} \,.
\end{equation*}
So the action of $S_{e_j}$ on $\Real^n$ is simply to change 
the sign of the $j$th coordinate. Thus, $\G$ is isomorphic to the $n$-fold 
direct product of the cyclic group of order 2, $(\Integer_2)^n$.
It then follows from \Cref{thm:solar_minimality} that 
\begin{equation*}
    U(x) = \argmin_\theta \frac{1}{2} \norm{x - \theta}^2 
         + \lambda \norm{\theta}_1
\end{equation*}
is computationally minimal for all $\ell_1$-penalized expofam-type estimators 
with generators $\phi$ that depend only the magnitude of the coordinates of 
$\theta$.
Since the nonnegative regression penalty has the same base as the $\ell_1$ norm, the same argument as above shows that
\begin{equation*}
    U(x) = \argmin_{\theta \geq 0} \frac{1}{2} \norm{x - \theta}^2 
\end{equation*}
is computationally minimal for all nonnegative constrained expofam-type estimators with generators $\phi$ that depend only the magnitude of the coordinates of $\theta$. These two cases recover Examples 5 and 7 of \citet{vu:group}.

\subsection{Fused lasso/total variation and isotonic regression}
\label{sec:fused_lasso}
The base vectors of the fused lasso penalty and 
isotonic regression are 
\begin{equation*}
    B = \parens{ (e_{j+1} - e_j) / \sqrt{2} \given j = 1, \ldots, n-1 } 
    \,.
\end{equation*}
As a generator of a reflection group, this set is known as a fundamental 
system for the root system of the permutation group $\PermGroup_n$ 
\citep[36]{kane:reflection}. Another way to see this is to write out the 
reflection
\begin{equation*}
    S_{e_{j+1} - e_{j}} \,.
\end{equation*}
As a matrix it is equal to the identity everywhere except in the 
2-by-2 submatrix for rows and columns $(j,j+1)$ where it has the form
\begin{equation*}
    \begin{bmatrix}
        0 & 1 \\
        1 & 0
    \end{bmatrix} \,.
\end{equation*}
So the action of $S_j$ on $\Real^n$ is the transposition $(j, j+1)$. It is 
well-known that these transpositions generate all permutations in 
$\PermGroup_n$. Then it follows from 
\Cref{thm:solar_minimality} that 
the least squares fused lasso (resp. isotonic regression) 
is computationally minimal for all total variation penalized (resp. isotonic 
regression) expofam-type estimators with permutation invariant generators.

\paragraph{Comparison with existing results}
An instance of this phenomenon has already been pointed out in the 
literature.  \citet{barlow.brunk:isotonic} considered 
generalized isotonic regression estimators of the form
\begin{equation}
  \label{eq:generalized_isotonic}
  \begin{aligned}
      &\text{minimize}
        &   &
        \sum_{j=1}^n [f(\theta_j) - x_j \theta_j] w_j
      \\
      &\text{subject to}
        &   & \theta_1 \leq \theta_2 \leq \cdots \leq \theta_n
        \,,
  \end{aligned}
\end{equation}
where $f$ is a proper convex function on $\Real$ and $w_j > 0$ are fixed weights. They showed that the solution to this problem could be obtained 
from the least squares solution with $f(u) = u^2/2$ \citep[Theorem 3.1]{barlow.brunk:isotonic}, i.e. that least squares is computationally minimal 
for generalized isotonic regression. There are two key differences with 
what we derived just now.  \Cref{eq:generalized_isotonic} is an expofam-type 
estimator with generator of the form 
\begin{equation*}
    \phi(\theta) = \sum_j f(\theta_j) w_j \,.
\end{equation*}
This is a \emph{separable} function and not necessarily permutation invariant 
unless the weights are constant. In constant weights case, this generator 
is permutation invariant, however our result allows possibly nonseparable 
permutation invariant functions.

\citet{dumbgen.kovac:extensions} studied generalizations of total variation 
denoising solving problems of the form 
\begin{equation}
  \begin{aligned}
      &\text{minimize}
        &   &
        \sum_{j=1}^n f_j(\theta_j) - x_j \theta_j 
        + \lambda \sum_{j=1}^{n-1} \abs{\theta_{j+1} - \theta_j}
        \,.
  \end{aligned}
\end{equation}
In the case where $f = f_j$, $j=1,\ldots,n$, is the log-partition of 
a one-dimensional exponential family they showed that the solution to 
\Cref{eq:generalized_tautstring} could be obtained from the least squares fit.
This is a special case of our result with 
$\phi(\theta) = f(\theta_1) + \cdots f(\theta_n)$, 
which is clearly a permutation invariant function.
Our result, however, does not require that the generator be separable.

\paragraph{Taut strings}
There is a well-known connection between 1d total variation denoising and 
taut strings \citep{mammen.:locally,davies.kovac:local}. 
Let $D_n : \Real^n \to \Real^{n-1}$ denote the first-difference operator 
on $\Real^n$, i.e. 
\begin{equation*}
    [D_n x]_i = x_{i+1} - x_i\,, \quad i = 1,\ldots,n-1\,.
\end{equation*}
With some algebra, we can express the dual feasible set as
\begin{equation}
    \label{eq:tautstring_transform}
    x - Z(B, \Lambda)
    =
    D_{n+1} \Sigma
    \,,
\end{equation}
where 
\begin{equation*}
    \Sigma = 
    \set*{z \in \Real^{n+1} 
    \given z_1 = w_1, z_{n+1} = w_{n+1}, \norm{z - w}_\infty \leq \lambda}
\end{equation*}
and $w \in \Real^{n+1}$ is the cumulative sum of $x$:
\begin{equation*}
    w_i = \sum_{j < i} x_j \,.
\end{equation*}
$\Sigma$ can be interpreted as a set of \emph{strings} constrained to a 
tube of radius $\lambda$ centered at $y$.  The end points of the string are 
fixed. The taut string problem is to find the string in $\Sigma$ of minimal 
length so that it is made taut: 
\begin{equation}
  \label{eq:generalized_tautstring}
  \begin{aligned}
      &\text{minimize}
        &   &
        \sum_j \sqrt{1 + (z_{j+1} - z_j)^2}
      \\
      &\text{subject to}
        &   & z \in \Sigma
        \,.
  \end{aligned}
\end{equation}
Using the identity \Cref{eq:tautstring_transform}, the taut string problem 
\Cref{eq:generalized_tautstring} can be written as
\begin{equation*}
    \min_{y \in x - Z(B,\Lambda)} \phi^*(y) \,,
\end{equation*}
where
\begin{equation*}
    \phi^*(y) = \sum_j \sqrt{1 + y_j^2} 
\end{equation*}
is permutation invariant convex function.  Therefore, by \Cref{lem:g_minimal_equivalance} and \Cref{thm:minimal_norm_solar},
the taut string solution 
is given by the $\G$-minimal element of $x - Z(B,\Lambda)$ which is the same as least squares solution:
\begin{equation*}
  \begin{aligned}
      &\text{minimize}
        &   &
        \sum_j (z_{j+1} - z_j)^2
      \\
      &\text{subject to}
        &   & z \in \Sigma
        \,.
  \end{aligned}
\end{equation*}
Additional insight can be gained by identifying the $\G$-majorization 
ordering. Since $\G$ is the permutation group, it is exactly the 
classical majorization ordering. So the taut string solution is the 
string with least majorized first difference.

\paragraph{Graph-guided versions}
We conclude this example by pointing out that similar results hold for the graph fused lasso and graph isotonic regression.
\begin{proposition}
\label{pro:graph_fused_permutation}
If $\mathcal{E}$ is edge set of a connected undirected graph on 
$\bracks{n}$ and
\begin{equation*}
    B = \parens[\big]{ (e_j - e_i) / \sqrt{2} 
        \given \set{i,j} \in \mathcal{E}}
    \,,
\end{equation*}
then $\G(B) = \PermGroup_n$.
\end{proposition}
See \Cref{sec:proofs} for the proof. A remarkable consequence of this 
observation is that the computational minimality phenomenon holds for the 
graph-guided fused lasso for as well.
\begin{corollary}
\label{cor:computational_minimality_fused_lasso}
Within the class of expofam-type estimators with fused lasso penalty 
(resp. isotonic regression) on a 
connected graph and permutation invariant generators $\phi \in \Gamma_0(\X)$, 
the least squares estimator is computationally minimal.
\end{corollary}
If the graph is not connected, then it is not hard to see that the resulting 
reflection group is the direct product of permutation groups on nodes of each 
of the connected components.  In other words, it is the subgroup of 
permutation generated by transpositions within each connected component. So 
the generators in the above corollary must invariant under permutations within those associated coordinates.

\subsection{Sparse fused lasso}
The sparse fused lasso penalty has the form
\begin{equation*}
  \lambda_1 \sum_j \abs{\theta_j} 
  + \lambda_2 \sum_j \abs{\theta_{j+1} - \theta_j}\,.
\end{equation*}
It encourages both sparsity and smoothness in estimates of $\theta$.
The base representation of this penalty has base vectors from the lasso 
and fused lasso penalties:
\begin{equation*}
  B = \parens[\big]{
    e_1,\ldots,e_n, 
    (e_2 - e_1)/\sqrt{2}, \ldots, (e_n - e_{n-1})/\sqrt{2} }
  \,.
\end{equation*}
So $\G(B)$ must contain the reflection groups corresponding to 
the lasso and fused lasso as subgroups.  
The effect of the reflections in these subgroups on the standard basis 
are 
\begin{equation*}
  S_{e_j} = \text{sign change on $e_j$}
\end{equation*}
and
\begin{equation*}
  S_{e_{j+1} - e_j} = \text{interchange of $e_j$ and $e_{j+1}$}
  \,.
\end{equation*}
So $\G(B)$ must be isomorphic to the semidirect product 
$(\Integer_2)^n \rtimes \PermGroup_n$ \citep[10]{kane:reflection}, 
and it acts on $\Real^n$ by sign change and permutations. 
Then by \Cref{thm:solar_minimality}, among expofam-type estimators 
with a sparse fused lasso penalty and generators invariant under sign 
changes and permutations, the penalized least squares estimator is 
computationally minimal.

\subsection{Trend filtering}
Our next example is more of a nonexample. The base vectors of the trend filtering penalty are 
\begin{equation*}
    B = \parens[\Big]{ (e_{j+2} - 2 e_{j+1} + e_j) / \sqrt{6} \given j=1,\ldots,n-2}
    \,.
\end{equation*}
In an attempt to classify the group $\G(B)$ we compute the angle between pairs 
of base vectors. In this way we obtain the half-angle of rotation of the 
composition of pairs of reflections $S_{r_i} S_{r_j}$:
\begin{equation*}
    \arccos( \innerp{r_i}{r_j} )
    = 
    \begin{cases}
           0 & \text{if $\abs{i-j} = 0$,} \\
        \arccos(2/3) & \text{if $\abs{i-j} = 1$,} \\
        \arccos(1/6) & \text{if $\abs{i-j} = 2$,} \\
           \pi/2 & \text{otherwise.}
    \end{cases}
\end{equation*}
Note that some of these angles are irrational multiples of $\pi$ 
\citep{varona:rational}, so the subgroups generated by each  
single rotation $S_{r_i} S_{r_j}$, $i=1$ or $i=2$, is of infinite order and 
isomorphic to the infinite dihedral group \citep[5]{dolgachev:reflection}. 
Since $\G(B) \subseteq \OrthoGroup(\X)$, 
perhaps the most practically useful statement we can make at the moment is that \Cref{thm:solar_minimality} is also true with $\G = \OrthoGroup(X)$. So 
least squares trend filtering is computationally minimal for the subcollection 
of expofam-type estimators with orthogonally invariant generators.

\subsection{General solar penalties}
\label{sec:general_solar_penalties}
As we have seen in the preceding examples, the reflection group generated 
by the base vectors $B = (r_1,\ldots,r_m)$ 
of a solar penalty depends on the arrangement of the hyperplanes 
$H_r$, $r \in B$.  The resulting group $\G(B)$ may or may not be finite. 
The general theory of reflection groups is beyond the scope of this paper, 
but the book by \citet{kane:reflection} is great reference. In the case of 
trend filtering, $\G(B)$ failed to be finite because 
\begin{equation*}
    \frac{1}{\pi} \arccos(\innerp{r_i}{r_j}) = q_{ij}
\end{equation*}
was not rational for all $i,j$.  We can easily verify that $q_{ij}$ 
is rational for all of the other examples.  This is necessary and sufficient 
for $\G(B)$ to be a finite reflection group 
\citep[see][Chapter 6]{kane:reflection}.
 
\section{Discussion}
\label{sec:discussion}
Our main example explains within the computational sufficiency framework 
a deep explanation for phenomena discovered by \citet{barlow.brunk:isotonic} 
and \citet{dumbgen.kovac:extensions} for isotonic regression and 1d total 
variation denoising, respectively.  As we mentioned in the introduction, this 
has implications for both computation and inference.  There is, however, some 
limits to what the existing theory can provide.  Our final example of trend 
filtering showed that in some cases, the reflection group associated with a 
solar penalty may not have an easy interpretation.  Nonetheless, we believe 
that the theory developed in this paper can have application beyond the 
examples considered.  Here we suggest three possible directions for future 
research.

\paragraph{Extending to the generalized group lasso}
An immediate question raised by this work is how to extend the results to 
generalized group lasso penalties.  The use of the word ``group'' here 
refers to grouping of a variables as in the original group lasso paper 
\citep{yuan.lin:model}.  This is important for the extension of fused lasso 
on a graph to multivariate observations and also for additive models \citep[see, e.g.,][]{petersen.witten.ea:fused}. One special case that has attracted 
much attention recently is the so-called convex fusion clustering 
\citep[e.g.,][]{hocking.joulin.ea:clusterpath} which is a convex relaxation of 
hierarchical clustering. One obvious way forward is to define a generalization 
of the solar penalty as a Minkowski sum of line segments, rays, and norm 
balls. Then most of the analysis in the paper should carry through, with 
block coordinate descent replacing coordinate descent.  The main challenge 
will be finding a good notation system and identifying the corresponding 
groups.

\paragraph{Constructing solar penalties from a given reflection group}
Rather than applying the theory to an existing generalized lasso penalty, 
we could instead start from some known reflection group for which we would 
like to maintain invariance. \Cref{sec:general_solar_penalties} suggests that 
we can construct a solar penalty by taking as base vectors the normal vectors 
of any generating set of reflections. This is related to concept of root 
systems and fundamental systems in the theory of reflection groups 
\citep[see][Chapters 2---3]{kane:reflection}.  Alternatively, we could try 
to perturb the base vectors of trend filtering to obtain a more friendly group 
while still retaining desirable properties of trend filtering.

\paragraph{Exploiting group structure and geometry to develop efficient algorithms}
The proof of \Cref{thm:minimal_norm_solar} using dual coordinate descent 
suggests that the reflection group $\G(B)$ has an 
intrinsic role in iterative optimization.  This is certainly the case 
for least squares regression problems.  The optimality conditions 
for the dual problem involve normal cones of the dual feasible set 
$x - Z(B,\Lambda)$.  These normal cones are related to the arrangement of 
hyperplanes corresponding to the base vectors and also to the 
Weyl chambers of the reflection group $\G(B)$.  Is there a way to use this 
group structure and geometry to develop a more efficient algorithm for solving 
the dual problem?
 
\section*{Acknowledgments}
This work was supported by the National Science Foundation under Grant No. 
DMS-1513621. The topic was inspired by conversations that took place during 
the Statistical Scalability programme at the Isaac Newton Institute for 
Mathematical Sciences.  Thanks the institute for its hospitality, and to 
Francis Bach and Ryan Tibshirani for their questions. Part of this research 
was completed while the author was visiting Keio University. Thanks to Kei 
Kobayashi for his hospitality. 
\appendix

\section{Additional proofs}
\label{sec:proofs}
\subsection{Proof of Lemma~\ref{lem:g_minimal_equivalance}}
\begin{proof}
\Implies{case:g_minimal}{case:g_minimizes_convex_functions}: 
By \Cref{lem:g_majorization_equivalence}, 
\begin{equation*}
    f(z) \leq f(z)
\end{equation*}
for all $z \in K$. Then
\begin{equation*}
    f(z) \leq \inf_{z \in K} f(z) \,,
\end{equation*}
and we must have equality because $z \in K$.

\Implies{case:g_minimizes_convex_functions}{case:g_minimal}: 
For any fixed $u \in \X$, the function 
\begin{align*}
    f_u(t)
    &= h_{\G\cdot t}(u) \\
    &= \sup_{g \in \G} \innerp{g \cdot t}{u} \\
    &= \sup_{g \in \G} \innerp{t}{g^{-1} \cdot u} 
\end{align*}
is $\G$-invariant and convex, 
because it is the pointwise supremum of a family of convex functions 
\citep[Proposition 8.16]{bauschke.combettes:convex}. Then for $z \in K$, 
\begin{align*}
    h_{\G\cdot y}(u) 
    &= f_u(y) \\
    &= \inf_{w \in K} f_u(w) \\
    &\leq f_u(z) = h_{\G\cdot z}(u) \,,
\end{align*}
and by \Cref{lem:g_majorization_equivalence}, $y \preceq_\G z$.
\end{proof}

\subsection{Proof of Lemma~\ref{lem:g_minimality_criteria}}
\begin{proof}
\ref{statement:g_minimality_trivial}: This is trivial.

\ref{statement:g_minimality_increasing}:
By \Cref{lem:g_majorization_equivalence}, 
if $y \preceq_\mathcal{H} z$, then 
\begin{equation*}
    y \in \conv(\mathcal{H} \cdot y) \subseteq \conv(\mathcal{G} \cdot z)
\end{equation*}
and $y \preceq_\G z$.

\ref{statement:g_minimality_orthogonal}: 
It is easy to see that 
$y \preceq_{\OrthoGroup(\X)} y$ if and only if $\norm{y} \leq \norm{z}$. 
So $y$ is $\OrthoGroup(\X)$-minimal in $K$ if and only if 
\begin{equation*}
    \norm{y} = \inf_{z \in K} \norm{z} \,.
\end{equation*}
Since $K$ is closed and convex, such an element exists uniquely.

\ref{statement:g_minimality_minimum_norm}: 
By \ref{statement:g_minimality_increasing} and 
\ref{statement:g_minimality_orthogonal}, if $y$ is $\G$-minimal in $K$, then 
it must be the unique $\OrthoGroup(\X)$-minimal element of $K$, which, 
as shown above, is the minimum norm element.
\end{proof}

\subsection{Proof of Theorem~\ref{thm:least_squares_minimal}}
\begin{proof}
We will prove the following claims:
\begin{enumerate}[label=(\arabic*),ref=(\arabic*)]
    \item \label{case:u_is_g_minimal}
    $U(x)$ is $\G$-minimal in $x - C$,
    \item \label{case:u_is_dual_optimal}
    $U(x)$ is dual optimal for every $T \in \mathcal{M}$
    \item \label{case:primal_recovery}
    \Cref{eq:primal_recovery} recovers the primal solutions, and
    \item \label{case:lse}
    $U \in \mathcal{M}$.
\end{enumerate}
\paragraph{Claim \ref{case:u_is_g_minimal}}
This is immediate from \Cref{thm:g_minimality_and_minimal_norm}.

\paragraph{Claim \ref{case:u_is_dual_optimal}}
Let $x \in \X$, $y_* = U(x)$, and fix $T \in \mathcal{M}$ with generator 
$\phi$.  The dual problem to $T(x)$ is 
\begin{equation*}
    - \min_{y \in x - C} \phi^*(y) \,.
\end{equation*}
Note that $\phi^*$ is $\G$-invariant and convex. 
If this dual problem has a solution, then it follows from 
the $\G$-minimality of $y_*$ in $x - C$ that $y_*$ is a solution.

\paragraph{Claim \ref{case:primal_recovery}}
\Cref{assumption:strong_duality} ensures that strong duality holds so that a primal 
and dual solution pair $(\theta, x - z)$ are related by the equations
\begin{align*}
    x - z &\in \partial \phi(\theta) \\
    \theta &\in N_C(z) \,.
\end{align*}
This in turn is equivalent to
\begin{align*}
    \innerp{x - z}{\theta} &= \phi^*(x - z) + \phi(\theta) \\
    \innerp{z}{\theta} &= h_C(\theta) \,.
\end{align*}
Thus,
\begin{equation*}
    T(x) 
    = \argmin_\theta \phi(\theta) - \innerp{y}{\theta} + h_C(\theta) \,.
\end{equation*}
for any dual solution $y$. Conversely, if $T(x)$ is nonempty 
then a dual solution exists, and by Claims 1 and 2, 
$y_*$ is a dual solution and \Cref{eq:primal_recovery} holds.

\paragraph{Claim \ref{case:lse}}
Let $\phi = \frac{1}{2} \norm{}^2$. 
\Cref{assumption:strong_duality} holds in this case and $\phi$ is 
$\G$-invariant, because $\G \subseteq \OrthoGroup(X)$. Then by applying \Cref{eq:primal_recovery}, 
\begin{equation*}
    \argmin_\theta \phi(\theta) - \innerp{U(x)}{\theta}
    = \argmin_\theta \frac{1}{2} \norm{\theta - U(x)}^2
    = U(x) \,.
\end{equation*}
So $U \in \mathcal{M}$ is computationally necessary and hence computationally minimal.
\end{proof}

\subsection{Proof of Proposition~\ref{pro:graph_fused_permutation}}
\begin{proof}
The reflections associated with $B$ are transpositions of the form 
$(i,j) \in \mathcal{E}$. Since the graph is connected, 
between any pair of vertices, say $u,v$, there exists a path. Then the 
product of the transpositions of the edges along the graph, taken in order 
from $u$ to $v$, is simply the transposition $(u,v)$.  For example,
\begin{equation*}
    (u, i_1) (i_1, i_2) (i_2, v)
    = (u, i_2) (i_2, v)
    = (u, v)
    \,.
\end{equation*}
Since $\G(B)$ contains all transpositions it must be $\PermGroup_n$.
\end{proof}  
\printbibliography
\end{document}